\documentclass[11pt]{article}

\usepackage{latexsym,amsfonts,amsmath,amsthm,graphics}
\usepackage{times}
\usepackage{epsfig}
\usepackage{algorithm}
\usepackage{bbm}
\usepackage{algpseudocode}
\usepackage{color} 
\usepackage[a4paper,dvips]{geometry}
\usepackage{tikz}
\usetikzlibrary{arrows.meta}
\usetikzlibrary{decorations.pathreplacing,angles,quotes,calligraphy}

\usepackage{capt-of}

\newtheorem{theorem}{{\sc Theorem}}
\newtheorem{lemma}{{\sc Lemma}}
\newtheorem{corollary}{{\sc Corollary}}
\newtheorem{remark}{{\sc Remark}}

\newtheorem{proposition}{{\sc Proposition}}

\newcommand{\RR}{\mathbb{R}}

\newcommand{\cS}{{\mathcal S}}
\newcommand{\cP}{{\mathcal P}}
\newcommand{\cL}{{\mathcal L}}

\newcommand{\x}{{\bf x}}
\newcommand{\bv}[1]{\mathbf{#1}}
\newcommand{\bx}{{\mathbf{x}}}
\newcommand{\vol}{ \mathrm{vol}}
\newcommand{\dd}{\mathrm{d}}

\begin{document}

\title{On the expected $\mathcal{L}_2$-discrepancy of stratified samples from parallel lines}

\author{Florian Pausinger}
\date{Queen's University Belfast, United Kingdom \\[3pt] \today}
\maketitle

\begin{abstract}
We study the expected $\mathcal{L}_2$-discrepancy of stratified samples generated from special equi-volume partitions of the unit square. The partitions are defined via parallel lines that are all orthogonal to the diagonal of the square. It is shown that the expected discrepancy of stratified samples derived from these partitions is a factor 2 smaller than the expected discrepancy of the same number of i.i.d uniformly distributed random points in the unit square. We conjecture that this is best possible among all partitions generated from parallel lines.\\
{\bf MSC2020:} 11K38, 05A18, 60C05 \hspace{1cm}{\bf Keywords:} Discrepancy; stratified sampling.
\end{abstract}

\section{Introduction}
\paragraph{Context.} Given a partition $\mathbf{ \Omega}=( \Omega_1, \ldots, \Omega_N)$ of $[0,1]^d$ into $N$ sets of positive Lebesgue measure.
A stratified sample, $P_{\mathbf{ \Omega}}$, is a set of $N$ points such that the $i$-th point is sampled independently and uniformly from the $i$-th set, $\Omega_i$, of the partition. Classical jittered sampling is an example of a stratified point set in which the unit cube is partitioned into $N=m^d$ axis-aligned cubes of volume $1/N$ for an integer $m$; we denote such as set with $P_{\mathrm{jit}}$.

The so-called $\mathcal{L}_2$-discrepancy of a finite set of points $\mathcal{P} = \left\{ \bx_1, \dots, \bx_N\right\}$ is a well-studied measure for the irregularities of the distribution of a point set. It is defined as
$$ \mathcal{L}_{2} (\mathcal{P}) := \left( \int_{[0,1]^d} \left| \frac{\#\left(\cP\cap[0, {\bv x}[\right)}{N} - \big|[0, \bv x[\big| \right|^2 \dd\bv x \right)^{1/2},$$
in which $\#\left(\cP\cap[0, {\bv x}[\right)$ counts the number of indices $1\leq i \leq N$ such that $\bx_i \in [0, {\bv x}[$,
and  $\big|[0, \bv x[\big|$ is the Lebesgue measure of 
$[0,\bv x[ :=\prod_{k=1}^d [0,x_k[$
with $\bv x = (x_1, \ldots, x_d)$; i.e. the $\cL_2$ norm of the so-called discrepancy function. Analogously, the more general $\mathcal{L}_p$-discrepancy is defined as the $\cL_p$ norm of the discrepancy function.
For an infinite sequence $\cS$ the $\cL_2$-discrepancy $\cL_{2}(\cS_N)$ is the $\cL_2$-discrepancy
of the first $N$ elements, $\cS_N$, of $\cS$.
We refer to the book \cite{DP} and the survey \cite{DP2} for further details. In particular, and in contrast to other measures such as the star-discrepancy, it is known how to construct deterministic point sets with the optimal order of magnitude of the $\cL_2$-discrepancy; see \cite{chen1,DP2,DP3}. 
For $d=2$ the optimal order of the $\cL_2$-discrepancy for finite point sets is known to be $\mathcal{O}(\sqrt{\log N}/N)$, which already goes back to a result of Davenport \cite{dave}.
The optimality of these constructions follows from a seminal result of Roth \cite{roth} who derived a general lower bound for the $\cL_2$-discrepancy of arbitrary sets of $N$ points in $[0,1]^d$; see e.g. \cite[Theorem 3.20]{DP}.

To put our results into context, the expected star discrepancy of a set $\cP_N$ of $N$ i.i.d.~uniform random points in $[0,1]^d$ is of order $\mathcal{O}(1/\sqrt{N})$ and as such independent of the dimension; see \cite{hnww} for the first upper bound, \cite{aist} for the first upper bound with explicit constant and \cite{doerr1} for the first lower bound as well as \cite{gnewuch} for the current state of the art results in this context.
An upper bound of this order was also calculated for the $\cL_2$-discrepancy in \cite{MKFP21}.
For two-dimensional point sets of $N$ i.i.d.~uniform random points, we thus also have an expected discrepancy of order $\mathcal{O}(1/\sqrt{N})$ similar to the two-dimensional regular grid (whose discrepancy is known to get worse as the dimension increases). 

\paragraph{Results. }
Our work is motivated by two results. First, the strong partition principle \cite[Theorem 1]{MKFP21} implies that stratified samples derived from an arbitrary equi-volume partition \emph{always} have a smaller expected $\cL_p$-discrepancy than random sets $\cP_N$.
Second, it is known that classical jittered sampling improves the order of magnitude of the expected discrepancy \cite{NF23}, i.e., we have for $N=m^d$ that
$$\mathbb{E}\mathcal{L}_2^2(\mathcal{P}_{\mathrm{jit}}) = \frac{1}{m^{2d}} \left[ \left( \frac{m}{2} \right)^d - \left( \frac{m}{2}-\frac{1}{6} \right)^d \right].$$
In particular, for $d=2$ and using Jensen's inequality, it follows that
$$\mathbb{E}(\mathcal{L}_2(\mathcal{P}_{\mathrm{jit}})) = \mathcal{O}\left( N^{-3/4}  \right).$$
While the strong partition principle tells us that any partition will improve the discrepancy, the result about jittered sampling says that certain partitions can actually make a huge difference. 
However, jittered sampling is restricted to sets with $N=m^d$ points and, in practice, it is often desired to have a construction that works for arbitrary $N$.

In \cite[Lemma 2]{MKFP21} it was shown that among all partitions of $[0,1]^2$ into two sets, the partition that splits the square into two triangles with the anti-diagonal being the dividing line minimizes the expected $\mathcal{L}_2$-discrepancy; see Figure \ref{fig:two}.

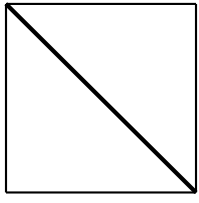
\begin{figure}[h!]
\begin{center}
\begin{tikzpicture}[scale=2.5]

\draw [thick] (4,0) -- (5,0);
\draw [thick] (5,1) -- (5,0);
\draw [thick] (4,1) -- (5,1);
\draw [thick] (4,0) -- (4,1);
\draw [ultra thick] (4,1) -- (5,0);

\end{tikzpicture}
\caption{The partition of the unit square into two convex sets with smallest expected $\mathcal{L}_2$-discrepancy} \label{fig:two}
\end{center}
\end{figure}

Motivated by this an infinite family of equivolume partitions, $\Omega_N$, was defined by hyperplanes orthogonal to the diagonal of the unit cube; see Figure \ref{fig:simplePartition} and Section \ref{sec:part} for details. We denote a stratified sample generated from $\Omega_N$ with $\mathcal{P}_{\Omega_N}$.
Based on extensive numerical results it was conjectured that
$$\lim_{N \rightarrow \infty} \frac{\mathbb{E}\mathcal{L}_2^2(\mathcal{P}_N)}{\mathbb{E}\mathcal{L}_2^2(\mathcal{P}_{\Omega_N})} = \frac{d}{d-1}.$$
For $d=2$ this is due to Kiderlen \& Pausinger \cite[Conjecture 2]{MKFP21}. The general case is due to Kirk \cite[Conjecture 3.8]{NThesis}. While the general case seems out of reach at the moment, the main result of this note is to prove this conjecture for $d=2$.

\begin{theorem} \label{thm:main}
Let $N$ be even and let $\Omega_N$ be the equivolume partition of $[0,1]^2$ defined via the generating set \eqref{genSet}, then
$$\mathbb{E}\mathcal{L}_2^2(\mathcal{P}_{\Omega_N}) = \frac{5}{72N} + \mathcal{O}(N^{-3/2})$$
\end{theorem}
 
The main ingredients of our proof are a method for the calculation of the discrepancy of stratified samples \cite[Proposition 3]{MKFP22}, the Euler-MacLaurin formula \cite{apostol} as well as a formula for the calculation of the volume of polygons/polytopes given as the intersection of half-spaces taken from \cite{cho}.

\begin{remark}
We restricted our proof to the case of $N$ even for simplicity and comment on the case $N$ odd at the end of the paper.
\end{remark}
 
\begin{corollary} Let $N$ be even and let $\Omega_N$ be the equivolume partition of $[0,1]^2$ defined via the generating set \eqref{genSet}, then (by Jensen's inequality)
$$\mathbb{E}\mathcal{L}_2(\mathcal{P}_{\Omega_N}) \leq \sqrt{\frac{5}{72N}} + \mathcal{O}\left(N^{-3/4}\right). $$
\end{corollary} 
\noindent Moreover, it was calculated in \cite{MKFP21} that
$$\mathbb{E}\mathcal{L}_2^2(\mathcal{P}_N) = \left( \frac{1}{2^d} - \frac{1}{3^d} \right) \frac{1}{N},$$
and, in particular, for $d=2$ and arbitrary $N$
$$\mathbb{E}\mathcal{L}_2^2(\mathcal{P}_N) = \left( \frac{1}{4} - \frac{1}{9} \right) \frac{1}{N} = \frac{5}{36N}.$$
From this we immediately get:

\begin{corollary} \label{cor:conj} Let $d=2$, and let $N$ be even, then
$$\lim_{N \rightarrow \infty} \frac{\mathbb{E}\mathcal{L}_2^2(\mathcal{P}_N)}{\mathbb{E}\mathcal{L}_2^2(\mathcal{P}_{\Omega_N})} = 2.$$
\end{corollary} 
\begin{remark}
We can compare our result to \cite[Example 1]{MKFP21} in which the discrepancy of equivolume partitions generated from vertical lines was studied. In contrast to our result, the calculation is almost trivial in this case and we get that
 $$
 \mathbb{E} {\cL_2}(\cP_{\mathrm{vert}})^2=\frac{3N+2}{36N^2}\approx \frac35 \mathbb{E} {\cL_2}(\cP_N)^2.
 $$
The interest in our result, besides the modest gain over this simple construction, is that we believe the factor 2 is best-possible for partitions constructed from parallel lines. This belief is both motivated by numerical experiments as well as the before mentioned \cite[Lemma 2]{MKFP21}.
\end{remark}

We close this introduction with various problems for future research:
\begin{enumerate}
\item Prove that our result is best possible among all partitions constructed from parallel lines.
\item Prove the conjecture for arbitrary dimension $d$.
\item Is there a general construction that gives a partition for arbitrary number of points $N$, that improves the asymptotic order of the discrepancy (as jittered sampling does for very particular values of $N$) and not only the constant factor?
\end{enumerate}

\paragraph{Outlook. } In Section \ref{sec:prelim} we define our partitions and introduce important tools for the calculation of the discrepancy of the stratified samples. Section \ref{sec:calculations} provides various important building blocks needed for the derivation of the general formula and it is shown how to approximate the discrepancy numerically. In Section \ref{sec:formula} we derive a general formula for the discrepancy. This formula is then simplified in Section \ref{sec:simple} which allows to prove Theorem \ref{thm:main} and its corollaries.

\section{Preliminaries} \label{sec:prelim}

\subsection{Partitions} \label{sec:part}
Given the unit square $[0,1]^2$ we would like to arrange $N-1$ parallel lines $H_{i}$ with $i=1, \ldots, N-1$, which are orthogonal to the main diagonal, $D$, of the square, such that we obtain an equivolume partition of the unit square; see Figure \ref{fig:simplePartition}.

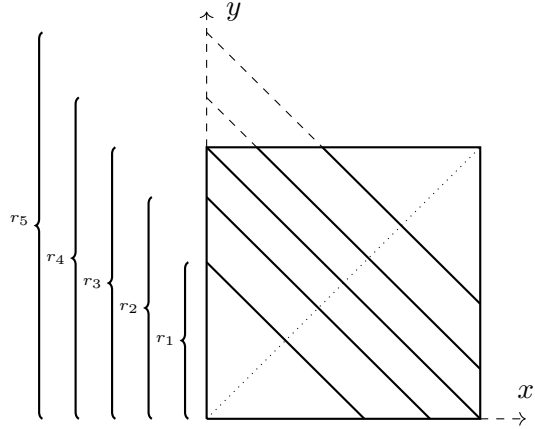
\begin{figure}[h]
\begin{center}
\begin{tikzpicture}[scale=1.2]
		
		\draw[thick] (0,0) -- (3,0)-- (3,3) -- (0,3)-- (0,0);
		\draw[dotted] (0,0) -- (3,3);
		
		\draw[thick] (0,1.73) -- (1.73,0);
		\draw[thick] (0,2.45) -- (2.45,0);
		\draw[thick] (0,3) -- (3,0);
		\draw[thick] (3-1.73,3) -- (3,3-1.73);
		\draw[thick] (3-2.45,3) -- (3,3-2.45);
		
		\draw[dashed] (0,3.55) --(3-2.45,3);
		\draw[dashed] (0,6-1.73) --(3-1.73,3);
		\draw[dashed,->] (0,3) --(0,4.5);
		\draw[dashed, ->] (3,0) -- (3.5,0);
		
		
		
		\node at (-0.45, 0.865) {\tiny $r_1$};
		\node at (-0.85, 1.225) {\tiny $r_2$};
		
		\node at (-1.25, 1.5) {\tiny $r_3$};
		\node at (-1.65, 1.77) {\tiny $r_4$};
		\node at (-2.05, 2.2) {\tiny $r_5$};
		
		\node at (0.3,4.5) {$y$};
		\node at (3.5,0.3) {$x$};
		
		\draw [decorate, decoration = {brace}, thick] (-0.2,0) --  (-0.2,1.73);
		\draw [decorate, decoration = {brace}, thick] (-0.6,0) --  (-0.6,2.45);
		
		\draw [decorate, decoration = {brace}, thick] (-1,0) --  (-1,3);
		\draw [decorate, decoration = {brace}, thick] (-1.4,0) --  (-1.4,3.55);
		\draw [decorate, decoration = {brace}, thick] (-1.8,0) --  (-1.8,6-1.73);
\end{tikzpicture}
\end{center}
\caption{Partition of the unit cube into $N=6$ equivolume slices that are orthogonal to the diagonal.} \label{fig:simplePartition}
\end{figure}

We denote the intersection $H_i \cap D$ of a line with the diagonal with $p_i$.
It is straightforward to calculate all points $p_i$ for arbitrary $N$. In fact, note that $H_i$ splits the unit square into two sets of volume $i/N$ and $1-i/N$. If $i\leq N/2$, we just need to look at the isosceles right triangle that $H_i$ forms with $(0,0)$. We know that this triangle has volume $i/N$. Denote the length of the two equal sides with $r_i>0$, then $r_i^2/2=i/N$. 
Therefore, we get that
\begin{equation} \label{twoDim}
r_i= \sqrt{\frac{2i}{N}}, 
\end{equation}
for all $i\leq N/2$. By symmetry, we also get the points $r_i$ with $i>N/2$.

To state this problem formally, we introduce further notation. Let $H_r^+$ be the positive half space defined as the set of all $\bx \in \RR^2$ satisfying
\begin{equation*} x_1 + x_2 - r \geq 0; \end{equation*}
accordingly let  $H_r^-$ be the corresponding negative half space and let $H_r$ be the line of all points $\bx \in \RR^2$ satisying $x_1 + x_2  + r = 0$.
For a given $N$, we would like to find the points $0<r_1 < \ldots < r_{N-1} < 2$ such that the corresponding lines $H_{r_1}, \ldots, H_{r_{N-1}}$ define a partition of $[0,1]^2$ into $N$ equivolume sets. We call the set 
$$\cS(N,2):=\{r_i : i=1, \ldots, N-1\}$$ 
the \emph{generating set} of the partition. Using \eqref{twoDim} we see that
\begin{align} \label{genSet}
\cS(N,2) = \left\{ \sqrt{\frac{2i}{N}}: 1\leq i \leq N/2 \right\} \cup \left\{2 -\sqrt{\frac{2(N-i)}{N}}: N/2 +1 \leq i < N \right\}.
\end{align}
Note that for even $N$ the value $1$ is in the set, while it is not for odd $N$.
In particular, we define each set in the partition $\mathbf{\Omega}_N=\{ \Omega_1, \ldots, \Omega_N\}$ as
\begin{equation*} \Omega_i = H_{r_{i-1}}^+ \cap H_{r_i}^- \cap [0,1]^2 \end{equation*}
for $i=1, \ldots, N$ with $r_{0}=0$ and $r_N=2$.

\subsection{Calculation of discrepancy}
Our work utilises a proposition from \cite{MKFP22} (see also \cite{MKFP21}) regarding the expected discrepancy of stratified samples obtained from an equivolume partition of the cube. 

\begin{proposition} [Proposition 3, \cite{MKFP22}]
	If $\boldsymbol{\Omega} = \{\Omega_1, \ldots, \Omega_N\}$ is an equivolume partition of a compact convex set $K \subset \mathbb{R}^d$ with $|K| > 0$, then 
	\begin{equation}\label{eq:originalequation}
		\mathbb{E}\mathcal{L}_2^2(\mathcal{P}_{\boldsymbol{\Omega}}) = \frac{1}{N^2 |K|} \sum_{i=1}^N \int_{K} q_i (\mathbf{x}) \left( 1 - q_i (\mathbf{x}) \right) \dd \mathbf{x}
	\end{equation} 
	with $q_i (\mathbf{x}) = \frac{\left| \Omega_i \cap [0, \mathbf{x}) \right|}{\left| \Omega_i \right|}.$
\end{proposition}

In our particular problem, we have that $\left| \Omega_i \right|= 1/N$ and $|K|=1$. Hence, we need to find
\begin{equation}\label{equ:q} q_i (\mathbf{x}) = N \cdot \left| \Omega_i \cap [0, \mathbf{x}) \right| \end{equation}
for every rectangle $[0, \mathbf{x})$ with $\mathbf{x} \in [0,1]^2$ to calculate the discrepancy.

\section{Towards a general formula} \label{sec:calculations}

The aim of this section is to derive a general formula for $q_i (\mathbf{x})$ as defined in \eqref{equ:q}. In a first step we explore how to calculate the area of intersection of a rectangle and a positive halfspace. In a second step we use this insight to state a formula for arbitrary $q_i (\mathbf{x})$. This is already enough to obtain a numerical approximation of the expected discrepancy of our stratified samples as shown in the final part of this section.

\subsection{Calculating areas of intersections}
Our first goal is to derive a general formula for the intersection of an arbitrary rectangle with a positive half-space; see Figure \ref{fig:intersection}.
This formula is our main tool to calculate the $q_i(\mathbf{x})$ for arbitrary $i$ and $\mathbf{x}$ and will enable us to calculate the expected discrepancy of our partitions.

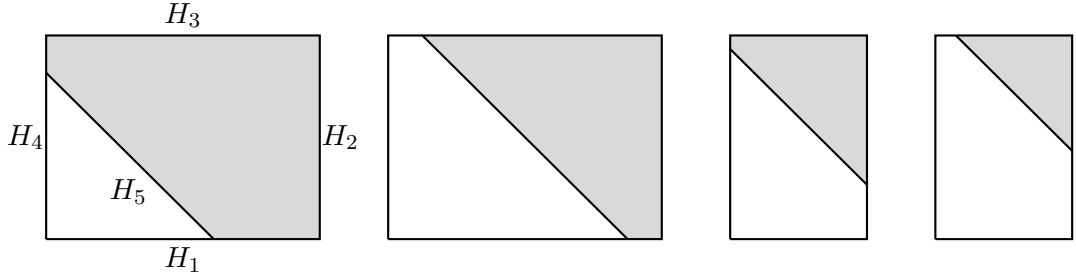
\begin{figure}[h]
\begin{tikzpicture}[scale=0.9]
		
		\fill[gray!30] (0,2.45) -- (2.45,0) -- (4,0)-- (4,3) -- (0,3)--(0,2.45);
		\fill[gray!30] (5.5,3) -- (8.5,0) -- (9,0)-- (9,3) -- (5.5,3);
		
		\fill[gray!30] (10,3) -- (10,2.8) -- (12,0.8)-- (12,3) -- (10,3);
		\fill[gray!30] (15-1.7,3) -- (15,3-1.7) -- (15,0)-- (15,3) -- (13,3)--(15-1.7,3);
		
		\draw[thick] (0,0) -- (4,0)-- (4,3) -- (0,3)-- (0,0);
		\draw[thick] (0,2.45) -- (2.45,0);
		
		\draw[thick] (5,0) -- (9,0)-- (9,3) -- (5,3)-- (5,0);
		\draw[thick] (5.5,3) -- (8.5,0);
		
		\draw[thick] (10,0) -- (12,0)-- (12,3) -- (10,3)-- (10,0);
		\draw[thick] (10,2.8) -- (12,0.8);
		
		\draw[thick] (13,0) -- (15,0)-- (15,3) -- (13,3)-- (13,0);
		\draw[thick] (15-1.7,3) -- (15,3-1.7);
		
		\node at (-0.3,1.5) {$H_4$};
		\node at (2,-0.3) {$H_1$};
		\node at (4.3,1.5) {$H_2$};
		\node at (2,3.3) {$H_3$};
		\node at (1.2, 0.7) {$H_5$};
\end{tikzpicture}
\caption{Intersections of a rectangle with a positive halfspace.} \label{fig:intersection}
\end{figure}

The lemma below follows directly from the general formula for the volume of the intersection of a $d$-dimensional polytope and a positive hyperplane as given in \cite[Theorem 7]{cho}. Let $P=[0,x] \times [0,y] = \bigcap_{i=1}^4 H_i^+$ be a rectangle in $\RR^2$ defined as the intersection of the four lines
\begin{align*}
H_1^+ &:= \{ (u,v): g_1(u,v) \geq 0\} \quad \text{ for } \quad g_1(u,v)=v \\
H_2^+ &:= \{ (u,v): g_2(u,v) \geq 0\} \quad \text{ for } \quad g_2(u,v)=-u+x \\
H_3^+ &:= \{ (u,v): g_3(u,v) \geq 0\} \quad \text{ for } \quad g_3(u,v)=-v + y \\
H_4^+ &:= \{ (u,v): g_4(u,v) \geq 0\} \quad \text{ for } \quad g_4(u,v)=u.
\end{align*}
Moreover, for given $r>0$ let
$$ H_5^+ := \{ (u,v): g_5(u,v) \geq 0\} \quad \text{ for } \quad g_5(u,v)=u+v-r.$$
We also define a matrix $A$, such that row $i$ contains the coefficients of $u$ and $v$ of $g_i$, i.e.,
$$A= \left( \begin{matrix}
0 & 1\\
-1 & 0 \\
0 & -1 \\
1 &0 \\
1 & 1 
\end{matrix} \right).
$$
We denote $2\times 2$ submatrices of $A$ as follows
$$(A)_{\{k,l \}}^{\{i,j \}}= \left( \begin{matrix}
a_{i,k} & a_{i,l}\\
a_{j,k} & a_{j,l}
\end{matrix} \right).
$$
The determinant is then given by 
$$A_{\{k,l \}}^{\{i,j \}}= \left| \begin{matrix}
a_{i,k} & a_{i,l}\\
a_{j,k} & a_{j,l}
\end{matrix} \right| = a_{i,k}a_{j,l} - a_{j,k}a_{i,l}.
$$
Now let $I \subset \{1,2,3,4\}$ and define
$$H_I := \bigcap_{j \in I} H_j \cap \bigcap_{i \in \{1,2,3,4 \}\setminus I} H_i^+ \cap (H_5^+ \setminus H_5); $$
see Figure \ref{fig:HI} for an illustration.

\begin{figure}[h]
\begin{center}
\begin{tikzpicture}[scale=0.9]
		\fill[gray!30] (5.5,3) -- (8.5,0) -- (9,0)-- (9,3) -- (5.5,3);
		
		\draw[thick, dashed, {Circle[open]}-{Circle[open]}] (5,-0.1)--(5,3.1);
		
		\draw[thick, dashed] (5,0) -- (8.5,0)-- (5.5,3) -- (5,3);
		\draw[thick, gray] (5.5,3) -- (9,3)--(9,0);
		\draw[thick, gray, -{Circle[open]}](8.5,0)--(9.1,0);
		
		\node at (9,3) {$\bullet$};
		
		\node at (4.7,1.5) {$H_4$};
		\node at (7,-0.3) {$H_1$};
		\node at (9.3,1.5) {$H_2$};
		\node at (7,3.3) {$H_3$};
		\node at (6.5, 1.2) {$H_5$};
		
		\node at (4.7, -0.3) {$A$};
		\node at (9.3, -0.3) {$B$};
		\node at (9.3,3.3) {$C$};
		\node at (4.7, 3.3) {$D$};
\end{tikzpicture}
\caption{The intersection of the five positive half-spaces generated by the lines $H_i$ with $1\leq i \leq 5$ without $H_5$ is shaded in gray. The corresponding set $H_I$ for $I=\{2,3\}$ consists of the vertex $C$.} \label{fig:HI}
\end{center}
\end{figure}
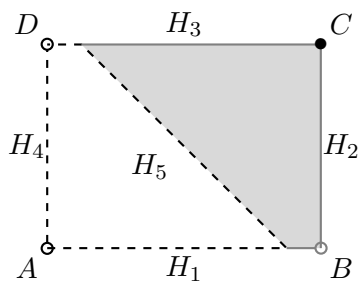

\begin{lemma} \label{lem:vol}
Let $P= \bigcap_{i=1}^4 H_i^+$, then we have that
$$\vol(P \cap H_5^+) = \sum_{I\subset \{1,2,3,4 \}}^{|I|=2} \sum_{\mathbf{v} \in H_I} \frac{(-1)^3 \left( g_5(\mathbf{v} ) A_{\{1,2\}}^I \right)^2 }{2 \left| A_{\{1,2\}}^I \right | \prod_{t \in I} A_{\{1,2\}}^{I \setminus t\cup \{5\} }}$$
\end{lemma}

\begin{proof}
This follows directly from \cite[Theorem 3.2]{cho}.
\end{proof}

Importantly, the second summation consists of either an empty summand or only one summand. In fact, all that matters is which of the four vertices of the rectangle $P=[0,x] \times [0,y]$ are contained in the intersection; see Figure \ref{fig:intersection} and Figure \ref{fig:HI}. We denote the vertices with $A, B, C, D$ such that 
\begin{align*}
A=H_4 \cap H_1=(0,0), \quad &B=H_1 \cap H_2=(x,0), \\
C=H_2 \cap H_3=(x,y), \quad &D=H_3 \cap H_4=(0,y).
\end{align*}
Hence, we can distinguish four different cases in which $S:=\bigcap_{i \in \{1,2,3,4 \}\setminus I} H_i^+ \cap (H_5^+ \setminus H_5)$:
\begin{itemize}
\item Case 1: If exactly $B, C, D \in S$ then $$\vol(P \cap H_5^+)=\frac{g_5(x,y)^2 -g_5(x,0)^2 - g_5(0,y)^2 }{2}.$$
\item Case 2: If exactly $B, C \in S$ then $$\vol(P \cap H_5^+)=\frac{g_5(x,y)^2 -g_5(x,0)^2}{2}.$$
\item Case 3: If exactly $C, D \in S$ then $$\vol(P \cap H_5^+)=\frac{g_5(x,y)^2 - g_5(0,y)^2 }{2}.$$
\item Case 4: If exactly $C \in S$ then $$\vol(P \cap H_5^+)=\frac{g_5(x,y)^2}{2}.$$
\end{itemize}
Since $A$ is $(0,0)$ it is never contained in $H_5^+\setminus H_5$. Therefore, these are indeed all cases to consider.

\subsection{Formula for $q_i$} \label{sec:qi}
The next aim is to write down general formulas for the $q_i$ and fixed $N$. Recall that
$$q_i (\mathbf{x}) = N \cdot \left| \Omega_i \cap [0, \mathbf{x}) \right|$$ 
for every rectangle $[0, \mathbf{x})$ with $\mathbf{x} \in [0,1]^2$.
Note that
$$ \Omega_i \cap [0, \mathbf{x}) = H_{5,i-1}^+ \cap H_{5,i}^- \cap \bigcap_{i=1,2,3,4} H_i^+,$$
in which $H_{5,i}$ refers to the $i$-th line defined via $g_{5,i}(x,y)=x+y-r_i$.
Lemma \ref{lem:vol} will be our main tool and we define
$$V_i=\vol \left( \bigcap_{i=1,2,3,4} H_i^+ \cap H_{5,i}^+ \right)=\vol([0, \mathbf{x}] \cap H_{5,i}^+).$$
With this it is easy to see that we have
\begin{itemize}
\item $q_1(x,y)= \begin{cases} N xy, &\text{if } (x,y) \in \Omega_1 \\ N (xy - V_1), & \text{if } (x,y) \in \bigcup_{j=2}^{N}\Omega_j \end{cases}$
\item $q_N(x,y)= \begin{cases} 0, &\text{if } (x,y) \in \bigcup_{j=1}^{N-1}\Omega_j \\ N V_{N-1}, & \text{if }  (x,y) \in \Omega_4 \end{cases}$
\end{itemize}
And for $2\leq i \leq N-1$ we have
\begin{itemize}
\item $q_i(x,y)= \begin{cases} 0 &\text{if } (x,y) \in \bigcup_{j=1}^{i-1}\Omega_j \\ N V_{i-1}, &\text{if } (x,y) \in \Omega_i \\ N (V_{i-1}-V_i), & \text{if } (x,y) \in \bigcup_{j=i+1}^{N}\Omega_j \end{cases}$
\end{itemize}

\subsection{Example of how to calculate $q_i$ in practice} \label{sec:example}
To illustrate this definition we look at the particular case $N=4$ and $(x,y)=(0.4, 0.8)$ as shown in Figure \ref{fig:qi}. We have that $(x,y) \in \Omega_3$ with $xy=0.32$. To calculate the values of the $q_i$, $1\leq i \leq 4$, we first need to calculate $V_1$ and $V_2$. We have that
\begin{align*}
g_{5,1}(x,y) &= x + y - r_1=x + y - \frac{1}{\sqrt{2}}, \\
g_{5,2}(x,y) &= x + y - r_2=x + y - 1,
\end{align*}
and, therefore, by Lemma \ref{lem:vol}:
\begin{align*}
V_1 &=\vol([0, \mathbf{x}] \cap H_{5,1}^+) \overset{Case \ 3} = \frac{g_{5,1}(0.4,0.8)^2 - g_{5,1}(0,0.8)^2 }{2} = \frac{0.2429 - 0.0086}{2} = 0.11715, \\
V_2 &=\vol([0, \mathbf{x}] \cap H_{5,2}^+) \overset{Case \ 4} =  \frac{g_{5,2}(0.4,0.8)^2}{2} = \frac{0.04}{2}=0.02.
\end{align*}
From this we can calculate
\begin{align*}
q_1 (0.4,0.8)&= 4 (0.32 - 0.11715) \approx 0.8114 \\ 
q_2 (0.4,0.8) &= 4 (0.11715 - 0.02) \approx 0.3886 \\
q_3  (0.4,0.8)& = 4\cdot 0.02 = 0.08 \\
q_4  (0.4,0.8)&= 0
\end{align*}

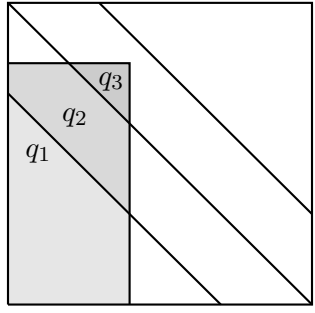
\begin{figure}[h]
\begin{center}
\begin{tikzpicture}[scale=0.8]
		
		\fill[gray!20] (0,0) -- (2,0) -- (2,1.5)-- (0,3.5) -- (0,0);
		\fill[gray!30] (2,1.5) -- (2,3) -- (1,4)-- (0,4) -- (0,3.5)--(2,1.5);
		\fill[gray!40] (2,4) -- (1,4) -- (2,3)-- (2,4);
		
		\draw[thick] (0,0) -- (5,0)--(5,5) -- (0,5) -- (0,0);
		\draw[thick] (5,0) -- (0,5);
		\draw[thick] (0,3.5) -- (3.5,0);
		\draw[thick] (1.5,5) -- (5,1.5);
		
		\draw[thick] (0,4) -- (2,4) -- (2,0);
		
		\node at (0.5,2.5) {$q_1$};
		\node at (1.1,3.1) {$q_2$};
		\node at (1.7,3.7) {$q_3$};
		
	\end{tikzpicture}
\caption{Illustration of $q_i$ for $N=4$.} \label{fig:qi}
\end{center}
\end{figure}

\subsection{Numerical approximation of discrepancy}
The calculations of the previous sections enable us to write down an approximation for $\mathbb{E}\mathcal{L}_2^2(\mathcal{P}_{\boldsymbol{\Omega}})$. This formula can be easily implemented and evaluated for arbitrary, but even (!) $N$. In particular, we have that
\begin{equation} \label{form:approx}
\mathbb{E}\mathcal{L}_2^2(\mathcal{P}_{\boldsymbol{\Omega}}) \approx \frac{1}{N^2} \sum_{i=1}^N \left( \frac{1}{M} \sum_{h=1}^M q_i\left( \x_h \right) \cdot \left( 1 - q_i\left( \x_h \right) \right) \right),
\end{equation}
in which the double integral from \eqref{eq:originalequation} is approximated using a finite set $(\x_h)_{h=1}^M \subset [0,1]^2$ of integration nodes.
The values in Table \ref{tab:numValues} were calculated using a Halton sequence in bases 2 and 3 for $M=40000$.
We use the numerical results to provide further evidence towards Corollary \ref{cor:conj} and display our findings in Figure \ref{fig:conj}.

\begin{figure}
\centering
\begin{minipage}[t]{.45\textwidth}
\centering
\vspace{0pt}
\includegraphics[width=\textwidth]{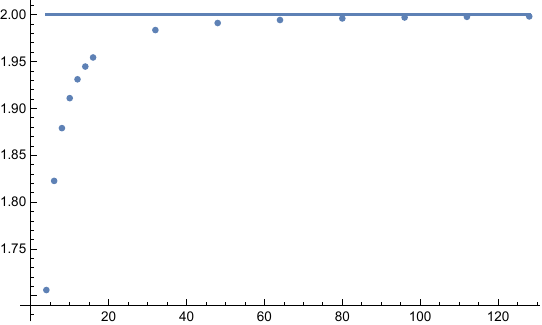}
\caption{Approximation of $\frac{\mathbb{E}\mathcal{L}_2^2(\mathcal{P}_N)}{\mathbb{E}\mathcal{L}_2^2(\mathcal{P}_{\Omega_N})}$ using $\widehat{\mathbb{E}}\mathcal{L}_2^2(\mathcal{P}_{\Omega_N})$ for different values of $N$. } \label{fig:conj}
\end{minipage}\hfill
\begin{minipage}[t]{.45\textwidth}
\centering
\vspace{0pt}
 \begin{tabular}{|c|c||c|c|}
    \hline
        $N$ & $\widehat{\mathbb{E}}\mathcal{L}_2^2(\mathcal{P}_{\boldsymbol{\Omega}})$ & $N$ & $\widehat{\mathbb{E}}\mathcal{L}_2^2(\mathcal{P}_{\boldsymbol{\Omega}})$ \\
        \hline
        4& 0.0203506& 32& 0.002188\\
         \hline
         6& 0.0127002&48&0.001453\\
         \hline
         8& 0.009239&64&0.001088\\
         \hline
        10 &0.007267 &80&0.000869\\
         \hline
        12 &0.005993&96&0.000724 \\
         \hline
       14  &0.005101 &112&0.000620\\
         \hline
      16   &0.004441&128&0.000543\\
         \hline
    \end{tabular}
    \captionof{table}{Numerical approximation of the expected discrepancy for different values of $N$.}
    \label{tab:numValues}
\end{minipage}
\end{figure}

\section{A general formula for the discrepancy} \label{sec:formula}
Based on the calculations from the previous section it is possible to derive a formula for $\mathbb{E}\mathcal{L}_2^2(\mathcal{P}_{\Omega_N})$. First, we can rewrite the expected discrepancy as follows:
\begin{align*}
\mathbb{E}\mathcal{L}_2^2(\mathcal{P}_{\boldsymbol{\Omega}}) &= \frac{1}{N^2} \sum_{i=1}^N \int_{[0,1]^2} q_i (\mathbf{x}) \left( 1 - q_i (\mathbf{x}) \right) \dd \mathbf{x} \\
&=\frac{1}{N^2}  \left(  \int_{[0,1]^2} \sum_{i=1}^N q_i (\mathbf{x}) d\mathbf{x} -   \sum_{i=1}^N  \int_{[0,1]^2} q_i^2 (\mathbf{x}) \dd \mathbf{x} \right)
	\end{align*} 

Looking at the definition of the $q_i$ it is easy to see that the sum in the first integral is a telescope sum such that the first integral reduces to
$$\int_{[0,1]^2} \sum_{i=1}^N q_i (\mathbf{x}) \dd \mathbf{x} = \int_{[0,1]^2} N xy \dd \mathbf{x} = N/4. $$

What remains are the integrals over the squares of the $q_i$.
We define 
\begin{equation} \label{def:Q} Q_i:=\int_0^1 \int_0^1 q_i^2(x,y) \dd y \dd x \end{equation}
and we recall that
$$g_{5,i}(x,y):=x+y-r_i.$$
The main idea of the subsequent calculations is to calculate $Q_i$ for every index $1\leq i \leq N$ based on the definition of $q_i$ stated in Section \ref{sec:qi}. The main challenge is that the definition of each $q_i$ in turn relies on the definition of the volumes $V_i$, i.e., of the intersections of the test box $[0,\bx]$ and the half space $H_{5,i}^+$. For each $i$, we subdivide the integration domain into subdomains such that $q_i(x,y)$ has the same functional form for all points $(x,y)$ in the same domain and, hence, it is possible to explicitly evaluate the (sub)integral.
The complexity of the overall calculation mainly comes from the different cases to consider while the integrals itself are straightforward to evaluate as the integrands are just polynomials.

\subsection{The case $i=1$}
First, we look at
$$Q_1=\int_0^1 \int_0^1 q_1^2(x,y) \dd y \dd x.$$
We have to split this integral into five subintegrals according to the definition of $V_1$; see Figure \ref{fig:q1} (Left) for an illustration. In particular, remember that the formula for $V_1$ depends on which vertices of the rectangle spanned by $x$ and $y$ lie in the intersection of the rectangle with $H_5^+$. Correspondingly we get
\begin{align*}
\int_0^1 \int_0^1 q_1^2(x,y) \dd y \dd x &= \int_0^{r_1} \int_0^1 q_1^2(x,y) \dd y \dd x + \int_{r_1}^1 \int_0^1 q_1^2(x,y) \dd y \dd x.
\end{align*}
Using the definition of $q_1$ we obtain
\begin{align*}
\int_0^{r_1} \int_0^1 q_1^2(x,y) \dd y \dd x = \int_0^{r_1} &\left( \int_0^{-x+r_1} N^2 x^2 y^2 \dd y + \int_{-x+r_1}^{r_1} \left( Nxy - \frac{N g_{5,1}(x,y)^2}{2} \right)^2 \dd y \right. \\
&\left. +\int_{r_1}^1 \left(N x y - \frac{N g_{5,1}(x,y)^2}{2} + \frac{N g_{5,1}(0,y)^2}{2} \right)^2 \dd y \right) \dd x,
\end{align*}
as well as
\begin{align*}
\int_{r_1}^1 \int_0^1 q_1^2(x,y) \dd y \dd x = \int_{r_1}^1 &\left( \int_{0}^{r_1} \left(N x y - \frac{N g_{5,1}(x,y)^2}{2} + \frac{N g_{5,1}(x,0)^2}{2} \right)^2 \dd y \right. \\
&\left. \int_{r_1}^1 \left(N x y - \frac{N g_{5,1}(x,y)^2}{2} + \frac{N g_{5,1}(x,0)^2}{2} + \frac{N g_{5,1}(0,y)^2}{2}\right)^2 \dd y\right) \dd x
\end{align*}
This can be simplified to
$$\int_0^1 \int_0^1 q_1^2(x,y) \dd y \dd x = 1-\frac{14 \sqrt{2}}{15 \sqrt{N}} + \frac{2}{5N}.$$

\begin{figure}[h]
\begin{center}
\begin{tikzpicture}[scale=0.8]
		\draw[thick] (0,0) -- (5,0)--(5,5) -- (0,5) -- (0,0);
		\draw[thick] (0,3.5) -- (3.5,0);
		\draw[thick, dashed] (0,3.5) -- (3.5,3.5) -- (3.5,0);
		\draw[thick, dashed] (3.5,5) -- (3.5,3.5) -- (5,3.5);
		\node at (-0.3,3.5) {\footnotesize{$r_1$}};
		\draw[dotted, thick,->] (5,0)--(6,0);
		\draw[dotted,thick,->] (0,5)--(0,6);
		\node at (6,0.3) {$x$};
		\node at (0.3,6) {$y$};
	\end{tikzpicture} \hspace{2cm}
	\begin{tikzpicture}[scale=0.8]
		
		\fill[gray!20] (3,5) -- (5,3) -- (5,5)--(3,5);
		\draw[thick] (0,0) -- (5,0)--(5,5) -- (0,5) -- (0,0);
		\draw[thick] (3,5) -- (5,3);
		\node at (3,5.3) {\footnotesize{$r_{N-1}-1$}};
		\draw[dotted, thick,->] (5,0)--(6,0);
		\draw[dotted,thick,->] (0,5)--(0,6);
		\node at (6,0.3) {$x$};
		\node at (0.3,6) {$y$};
	\end{tikzpicture}
\caption{Left: The case $i=1$ with its 5 subregions. Right: The case $i=N$. In this case only points $(x,y)$ in the upper right triangle have a non-zero contribution to the integral.} \label{fig:q1}
\end{center}
\end{figure}
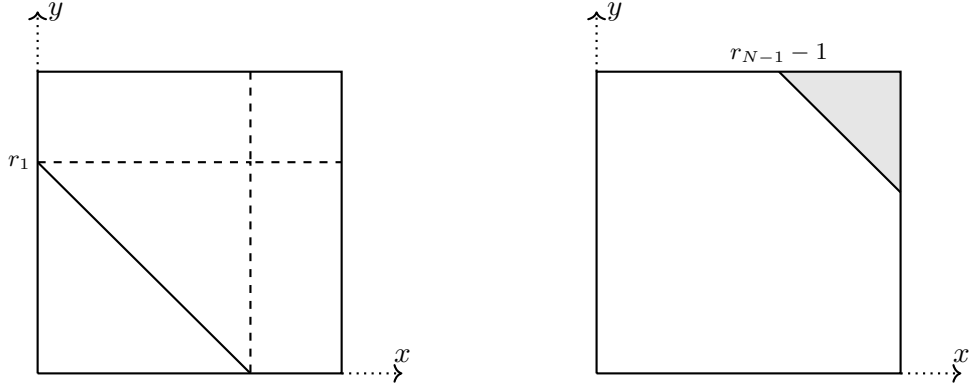

\subsection{The case $i=N$}
Next we study the case $i=N$ which is the simplest case. As illustrated in Figure \ref{fig:q1} (Right) only points in the upper right triangle contribute to the integral such that
\begin{align*}
Q_N=\int_0^1 \int_0^1 q_N^2(x,y) \dd y \dd x &= \int_{r_{N-1}-1}^1 \int_{-x+r_{N-1}}^1 \left( \frac{N g_{5,N-1}(x,y)^2}{2} \right )^2 \dd y \dd x \\
&= \frac{1}{15N}
\end{align*}

\subsection{The case $N/2 < i < N$}
The third case concerns indices $N/2 < i <N$. This case is relatively simple as well, because all the involved volumes of the intersection of the rectangle with a positive halfspace can be calculated via Case 4.
In particular, we have that 
\begin{align*}
\int_0^1 &\int_0^1 q_i^2(x,y) \dd y \dd x \\
&= \int_{r_{i-1}-1}^{r_i -1} \int_{-x+r_{i-1}}^1 \left( \frac{N g_{5,i-1}(x,y)^2}{2} \right )^2 \dd y \dd x \\ 
&+ \int_{r_{i}-1}^1 \left( \int_{-x+r_{i-1}}^{-x+r_i} \left( \frac{N g_{5,i-1}(x,y)^2}{2} \right )^2 \dd y + \int_{-x+r_1}^1 \left( \frac{N g_{5,i-1}(x,y)^2}{2} - \frac{N g_{5,i}(x,y)^2}{2}\right)^2 \dd y \right) \dd x
\end{align*}
Evaluating this gives
\begin{align*}
Q_i=\frac{4 i^3}{15 N}&+\frac{i^2 \left(4 N \sqrt{1-\frac{i}{N}} \sqrt{\frac{-i+N+1}{N}}-12
   N-2\right)}{15 N} \\
   &+\frac{i \left(-8 N^2 \sqrt{1-\frac{i}{N}} \sqrt{\frac{-i+N+1}{N}}+12 N^2+4
   N-3\right)}{15 N} \\
 & +\frac{4 N^3 \sqrt{1-\frac{i}{N}} \sqrt{\frac{-i+N+1}{N}}-4 N^3-2 N^2+3
   N+1}{15 N}
\end{align*}
See Figure \ref{fig:q2} (Right) for an illustration.

\begin{figure}[h]
\begin{center}
\begin{tikzpicture}[scale=0.8]

	\fill[gray!20] (0,4) -- (0,3) -- (3,0)--(4,0);
		\draw[thick] (0,0) -- (5,0)--(5,5) -- (0,5) -- (0,0);
		\draw[thick] (0,4) -- (4,0);
		\draw[thick] (0,3) -- (3,0);
		
		\draw[thick, dashed] (1,5) -- (1,0);
		\draw[thick, dashed] (3,5) -- (3,0);
		\draw[thick, dashed] (4,5) -- (4,0);
		
		\draw[thick, dashed] (0,3) -- (5,3);
		\draw[thick, dashed] (0,4) -- (5,4);

		\node at (-0.3,3) {\footnotesize{$r_{i-1}$}};
		\node at (-0.3,4) {\footnotesize{$r_{i}$}};
		
		\node at (1,-0.3) {\footnotesize{$r_{i}-r_{i-1}$}};
		\node at (3,-0.3) {\footnotesize{$r_{i-1}$}};
		\node at (4,-0.3) {\footnotesize{$r_{i}$}};
		\draw[dotted, thick,->] (5,0)--(6,0);
		\draw[dotted,thick,->] (0,5)--(0,6);
		\node at (6,0.3) {$x$};
		\node at (0.3,6) {$y$};
		
		\node at (0.5,-1) {\footnotesize{A}};
		\node at (2,-1) {\footnotesize{B}};
		\node at (3.5,-1) {\footnotesize{C}};
		\node at (4.5,-1) {\footnotesize{D}};
		
		\draw [decorate, decoration = {brace}, thick] (1,-0.6) --  (0,-0.6);
		\draw [decorate, decoration = {brace}, thick] (3,-0.6) --  (1,-0.6);
		\draw [decorate, decoration = {brace}, thick] (4,-0.6) --  (3,-0.6);
		\draw [decorate, decoration = {brace}, thick] (5,-0.6) --  (4,-0.6);

	\end{tikzpicture} \hspace{2cm}
	\begin{tikzpicture}[scale=0.8]
		\fill[gray!20] (3,5) -- (2,5) -- (5,2)--(5,3);
		\draw[thick] (0,0) -- (5,0)--(5,5) -- (0,5) -- (0,0);
		\draw[thick] (3,5) -- (5,3);
		\draw[thick] (2,5) -- (5,2);

		\draw[thick, dashed] (2,5) -- (2,0);
		\draw[thick, dashed] (3,5) -- (3,0);
		
		\node at (2,5.6) {\footnotesize{$r_{i-1}-1$}};
		\node at (3,5.3) {\footnotesize{$r_{i}-1$}};
		
		\node at (4,-0.3) {\footnotesize{\textcolor{white}{test}}}; 
		\draw[dotted, thick,->] (5,0)--(6,0);
		\draw[dotted,thick,->] (0,5)--(0,6);
		\node at (6,0.3) {$x$};
		\node at (0.3,6) {$y$};
		
		\node at (4.5,-1) {\footnotesize{\textcolor{white}{test}}};
		
	\end{tikzpicture}
\caption{Left: The case $2\leq i \leq N/2$ with an illustration of the four main cases A, B, C and D. Right: The case $N/2<i<N$.} \label{fig:q2}
\end{center}
\end{figure}
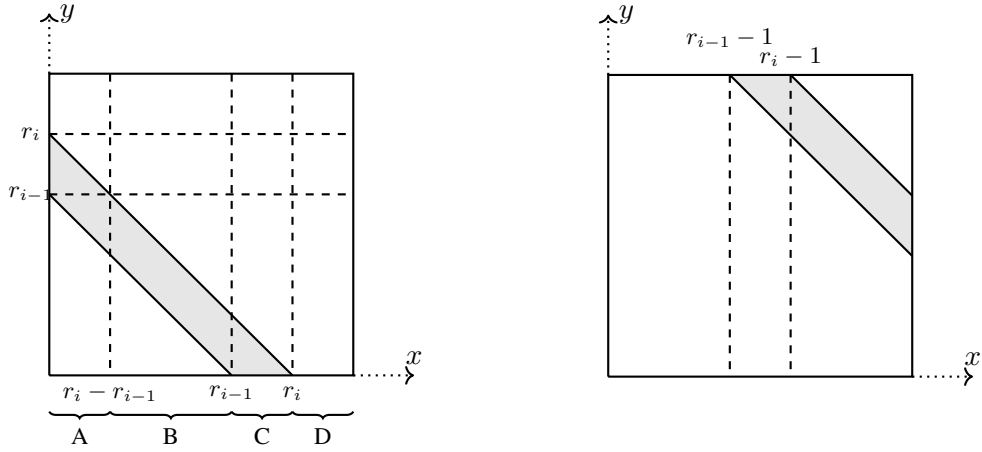

\subsection{The case $2 \leq i \leq N/2$}
The last and most complicated case concerns indices $2 \leq i \leq N/2$. We illustrate this case in Figure \ref{fig:q2} (Left). We split the outer integral over $x$ into four integrals, i.e.,
{\small
\begin{align*}
\int_0^1 &\int_0^1 q_i^2(x,y) \dd y \dd x \\
&=\underbrace{\int_0^{r_i-r_{i-1}} \int_0^1 \ldots }_{\text{Case A} }+ \underbrace{\int_{r_i-r_{i-1}}^{r_{i-1}} \int_0^1 \ldots}_{\text{Case B}} + \underbrace{\int_{r_{i-1}}^{r_i} \int_0^1 \ldots }_{\text{Case C} }+ \underbrace{\int_{r_i}^1 \int_0^1 \ldots}_{\text{Case D}},
\end{align*}
}
\noindent and consider each case separately in the following. In each case, we have to split the inner integral according to the different cases when computing $V_i(x,y)$ for a given index $i$ and a given point $(x,y)$.
The dashed lines in Figure \ref{fig:q2}(Left) together with the two solid boundaries of the gray strip, i.e., the $i$-th set in the partition, show the different cases to consider. The main motivation for this case distinction is the fact that the formula for $V_i$ is the same for all points $(x,y)$ belonging to the same region.

For convenience we define
$$\chi_i(x,y):= \frac{N g_{5,i}(x,y)^2}{2}$$

\paragraph{Case A.}

\begin{align*}
\int_0^{r_i-r_{i-1}} &\int_0^1 q_i^2(x,y) \dd y \dd x \\
&=\int_0^{r_i-r_{i-1}} \left( \int_{-x+r_{i-1}}^{r_{i-1}}  \chi_{i-1}(x,y)^2 \dd y  + 
\int_{r_{i-1}}^{-x+r_{i}}  \left(\chi_{i-1}(x,y)- \chi_{i-1}(0,y) \right )^2 \dd y \right. \\
& \quad + \int_{-x+r_{i}}^{r_{i}}  \left(\chi_{i-1}(x,y)- \chi_{i-1}(0,y) - \chi_i(x,y) \right )^2 \dd y \\
&\quad + \left. \int_{r_{i}}^1  \left(\chi_{i-1}(x,y)- \chi_{i-1}(0,y) - \chi_i(x,y) + - \chi_i(0,y)\right )^2 \dd y \right ) \dd x
\end{align*}

\paragraph{Case B.}

\begin{align*}
\int_{r_i-r_{i-1}}^{r_{i-1}} &\int_0^1 q_i^2(x,y) \dd y \dd x \\
&=\int_{r_i-r_{i-1}}^{r_{i-1}} \left( \int_{-x+r_{i-1}}^{-x+r_i} \chi_{i-1}(x,y)^2 \dd y + 
\int_{-x+r_{i}}^{r_{i-1}}  \left(\chi_{i-1}(x,y)- \chi_{i}(x,y) \right )^2 \dd y \right. \\
&\quad + \int_{r_{i-1}}^{r_i}  \left(\chi_{i-1}(x,y)- \chi_{i}(x,y) - \chi_{i-1}(0,y) \right )^2 \dd y \\
& \quad \left. +  \int_{r_i}^{1} \left( \chi_{i-1}(x,y)- \chi_{i}(x,y) - \chi_{i-1}(0,y) - \chi_{i}(0,y) \right)^2 \dd y\right)\dd x
\end{align*}

\paragraph{Case C.}

\begin{align*}
\int_{r_{i-1}}^{r_i} &\int_0^1 q_i^2(x,y) \dd y \dd x \\
&= \int_{r_{i-1}}^{r_{i}} \left(  \int_{0}^{-x+r_{i}}  \left(\chi_{i-1}(x,y)- \chi_{i-1}(x,0) \right )^2 \dd y \right. \\
&\quad + \int_{-x+r_{i}}^{r_{i-1}} \left(\chi_{i-1}(x,y)- \chi_{i-1}(x,0) - \chi_{i,x,y} \right )^2 \dd y \\
& \quad + \int_{r_{i-1}}^{r_{i}} \left(\chi_{i-1}(x,y)- \chi_{i-1}(x,0) - \chi_{i-1}(0,y)- \chi_i(x,y) \right )^2 \dd y \\
& \quad \left. +  \int_{r_i}^{1} \left( \chi_{i-1}(x,y)- \chi_{i-1}(x,0) - \chi_{i-1}(0,y)- \chi_i(x,y) + \chi_i(0,y) \right)^2 \dd y\right)\dd x
\end{align*}

\paragraph{Case D.}

\begin{align*}
\int_{r_i}^1 & \int_0^1 q_i^2(x,y) \dd y \dd x \\
&= \int_{r_i}^1 \left( \int_{0}^{r_{i-1}} \left( \chi_{i-1}(x,y)- \chi_{i-1}(x,0) - \chi_i(x,y) + \chi_i(x,0) \right)^2 \dd y \right. \\
& \quad + \int_{r_{i-1}}^{r_i} \left( \chi_{i-1}(x,y)- \chi_{i-1}(x,0) - \chi_{i-1}(0,y)- \chi_i(x,y) + \chi_i(x,0) \right)^2 \dd y \\
& \quad \left. + \int_{r_{i}}^{1} \left( \chi_{i-1}(x,y)- \chi_{i-1}(x,0) - \chi_{i-1}(0,y)- \chi_i(x,y) + \chi_i(x,0)+ \chi_i(0,y) \right)^2 \dd y \right ) \dd x\\
\end{align*}

\paragraph{All together.}
We can evaluate the expressions in the four cases and get the following formula:
\begin{align*}
Q_i= &\frac{1}{15N} \left( -4 i^3+i^2 \left(-16 \sqrt{2N} \sqrt{i-1}+4 \sqrt{(i-1)i} +16 \sqrt{2N} \sqrt{i}+10\right) \right. \\
   & \quad +i \left(32 \sqrt{2N}
   \sqrt{i-1}-8 \sqrt{(i-1)i} -40 \sqrt{2N}
   \sqrt{i}+5\right) \\
 & \quad \left.  -16 \sqrt{2N} \sqrt{i-1}+4
   \sqrt{(i-1)i}+10 \sqrt{2 N} \sqrt{i}+15 N-5 \right)
\end{align*}
 
\subsection{A general formula}

Fixing $N$ and substituting the values for the different indices $i$, we can now calculate the value of the discrepancy:
\begin{align} \label{form:main}
\mathbb{E}\mathcal{L}_2^2(\mathcal{P}_{\boldsymbol{\Omega}}) 
\notag &=\frac{1}{N^2}  \left(  \frac{N}{4} -   \sum_{i=1}^N  \int_{[0,1]^2} q_i^2 (\mathbf{x}) \dd \mathbf{x} \right) \\
 &= \frac{1}{4N} - \frac{1}{N^2} \left( 1-\frac{14 \sqrt{2}}{15 \sqrt{N}} + \frac{2}{5N} + \sum_{i=2}^{N-1} Q_i + \frac{1}{15N}\right)
\end{align} 
\begin{remark}
This gives a third way of calculating the expected discrepancy of our stratified samples, i.e., (i) we can generate random, stratified points sets, use Warnock's formula to evaluate them and average the discrepancy over $M$ such point sets, (ii) we can use the numerical approximation \eqref{form:approx}, or (iii) we can use \eqref{form:main} directly to calculate the expected value of the discrepancy. 
\end{remark}

However, the formula in \eqref{form:approx} is still very complicated and it is a non-trivial task to bring it into a form that allows us to prove Theorem \ref{thm:main}. This will be done in the next section.

\section{Simplifying the general formula} \label{sec:simple}

To simplify our main formula, we need accurate approximations of sums of a particular type. In the following, we first derive such approximations and apply them in a second step to simplify \eqref{form:main}. The simplified formula finally enables us to prove Theorem \ref{thm:main}.

\subsection{Approximating sums}
To derive an asymptotic estimate for our main formula, we need to approximate sums of the form
$$\sum_{i=2}^{n} i^k \sqrt{i-1},$$
for $k=0,\frac{1}{2}, 1, \frac{3}{2}, 2, \ldots$ up to the third leading term.
We use the Euler-MacLaurin formula \cite{apostol} to achieve this goal.

Let
$$ f(x) = x^k \sqrt{x-1}.$$
Applying the Euler-MacLaurin formula gives
\begin{align*}
 \sum_{i=2}^{n} f(i) -   \int_2^n f(x)\dd x &= \frac{f(2) + f(n)}{2} + \frac{1}{6} \frac{f'(n) - f'(2)}{2} +\frac{1}{6} \int_2^n f'''(x) P_3(x) \dd x
\end{align*}
where $P_3(x) = B_3(x- \left\lfloor x \right\rfloor)$ 
is the periodized third Bernoulli function. This function is, trivially, uniformly bounded and thus we have
$$ \left|  \int_2^n f'''(x) P_3(x) \dd x \right| \leq  \mbox{const} \cdot \int_2^n \left| f'''(x) \right| \dd x \leq \mbox{const} \cdot n^{k-3/2}.$$
It remains to evaluate the integral. This can be done in closed form but we can also get asymptotics. Note that
$$ \frac{\sqrt{x-1}}{\sqrt{x}} = \sqrt{1 - \frac{1}{x}} = 1 - \frac{1}{2x} - \frac{1}{8 x^2} - \frac{1}{16 x^3} - \dots$$
and thus
$$ \sqrt{x-1} = \sqrt{x} - \frac{1}{2\sqrt{x}} - \frac{1}{8 x^{3/2}} - \frac{1}{16 x^{5/2}} - \dots$$
which gets us
\begin{align*}
\int_2^n x^k \sqrt{x-1} \dd x = \int_2^n x^{k+1/2} \dd x - \frac{1}{2}  \int_2^n x^{k-1/2} \dd x  - \frac{1}{8}  \int_2^n x^{k-3/2} \dd x + \mathcal{O}(n^{k-3/2})
\end{align*}
and all these three integrals can be evaluated in closed form. We summarise this in the following lemma.
Note that we state the formulas in the subsequent lemma already in the form that we will use later, i.e., we sum indices $i$ with $2\leq i \leq n/2$.

\begin{lemma} \label{lem:sum}
Let $k=1/2$, then
$$\sum_{i=2}^{n/2} \sqrt{i} \sqrt{i-1} =
\frac{n^2}{8}-\frac{n}{4}+\frac{\frac{\sqrt{2} (6 n^2-10n-2)}{\sqrt{n-2} \sqrt{n}}+21}{24 \sqrt{2}}-\frac{1}{8} \log
\left(\frac{n}{4}\right)-1 + \mathcal{O}(n^{-1}) $$
Furthermore, let $k>0$ with $k \neq 1/2$, then
\begin{align*}
\sum_{i=2}^{n/2} i^k \sqrt{i-1} &= -\frac{2^{-k-\frac{5}{2}} \left(4^k-2 n^{k-\frac{1}{2}}\right)}{1-2
  k}-\frac{2^{\frac{1}{2}-k} n^{k+\frac{1}{2}}-2^{k+\frac{3}{2}}}{2 (2  k+1)} 
  +\frac{2^{-k-\frac{1}{2}} \left(n^{k+\frac{3}{2}}-2^{2 k+3}\right)}{2  k+3}\\
  &+\frac{1}{3} 2^{-k-3} \left(\frac{\sqrt{2} (2 k (n-2)+n (6 n-11)) n^{k-1}}{\sqrt{n-2}}+11\cdot 4^k-4^k k\right) + \mathcal{O}(n^{k-3/2})
  \end{align*}
\end{lemma}

\subsection{Generalised Harmonic Numbers}
The Euler-MacLaurin formula can also be applied to approximate generalised Harmonic numbers (also known as Faulhaber formulas if the exponent $k$ is a positive integer):
$$\sum_{i=1}^n i^k.$$
In particular, it can be shown that:
\begin{lemma} \label{lem:harmonic}
$$\sum_{i=1}^n i^k = \zeta(-k) + \frac{n^{k+1}}{1+k} + \frac{n^k}{2} + \frac{k\cdot n^{k-1}}{12} + \mathcal{O}(n^{k-2}),$$
in which $\zeta$ denotes the Riemann zeta function.
\end{lemma}

\subsection{A supporting function}
To simplify \eqref{form:main} we introduce several supporting functions, i.e., for $2\leq i \leq N$ we define
$$g(i):= Q_i + Q_{N+1-i}.$$
Using the general formulas, we can evaluate and simplify this and obtain
\begin{align*}
g(i)= -\frac{8 i^3}{15 N}&+\frac{i^2 \left(-16 \sqrt{2} N
   \sqrt{\frac{i-1}{N}}+8 N \sqrt{\frac{i-1}{N}}
   \sqrt{\frac{i}{N}}+16 \sqrt{2} N \sqrt{\frac{i}{N}}+20\right)}{15
   N} \\
   &+\frac{i \left(32 \sqrt{2} N \sqrt{\frac{i-1}{N}}-16 N
   \sqrt{\frac{i-1}{N}} \sqrt{\frac{i}{N}}-40 \sqrt{2} N
   \sqrt{\frac{i}{N}}\right)}{15 N} \\
   &+\frac{-16 \sqrt{2} N
   \sqrt{\frac{i-1}{N}}+8 N \sqrt{\frac{i-1}{N}}
   \sqrt{\frac{i}{N}}+10 \sqrt{2} N \sqrt{\frac{i}{N}}+15 N-5}{15 N}
 \end{align*}
Now, by definition we have
\begin{align*}
\sum_{i=2}^{N-1} Q_i &= \sum_{i=2}^{N/2} g(i) \\
&=- \sum_{i=2}^{N/2} \frac{8 i^3}{15 N} +\sum_{i=2}^{N/2} \frac{i^2 \left(-16 \sqrt{2} N
   \sqrt{\frac{i-1}{N}}+8 N \sqrt{\frac{i-1}{N}}
   \sqrt{\frac{i}{N}}+16 \sqrt{2} N \sqrt{\frac{i}{N}}+20\right)}{15
   N}  \\
  &+ \sum_{i=2}^{N/2} \frac{i \left(32 \sqrt{2} N \sqrt{\frac{i-1}{N}}-16 N
   \sqrt{\frac{i-1}{N}} \sqrt{\frac{i}{N}}-40 \sqrt{2} N
   \sqrt{\frac{i}{N}}\right)}{15 N} \\
   & +\sum_{i=2}^{N/2} \frac{-16 \sqrt{2} N
   \sqrt{\frac{i-1}{N}}+8 N \sqrt{\frac{i-1}{N}}
   \sqrt{\frac{i}{N}}+10 \sqrt{2} N \sqrt{\frac{i}{N}}+15 N-5}{15 N} \\
   &= S_3 + S_2 + S_1 + S_0,
\end{align*}
in which
\begin{align*}
S_3&:= -\sum_{i=2}^{N/2} \frac{8 i^3}{15 N}\\
S_2&:=\sum_{i=2}^{N/2} \frac{i^2 \left(-16 \sqrt{2} N
   \sqrt{\frac{i-1}{N}}+8 N \sqrt{\frac{i-1}{N}}
   \sqrt{\frac{i}{N}}+16 \sqrt{2} N \sqrt{\frac{i}{N}}+20\right)}{15
   N}\\
S_1&:=\sum_{i=2}^{N/2} \frac{i \left(32 \sqrt{2} N \sqrt{\frac{i-1}{N}}-16 N
   \sqrt{\frac{i-1}{N}} \sqrt{\frac{i}{N}}-40 \sqrt{2} N
   \sqrt{\frac{i}{N}}\right)}{15 N} \\
S_0&:=\sum_{i=2}^{N/2} \frac{-16 \sqrt{2} N
   \sqrt{\frac{i-1}{N}}+8 N \sqrt{\frac{i-1}{N}}
   \sqrt{\frac{i}{N}}+10 \sqrt{2} N \sqrt{\frac{i}{N}}+15 N-5}{15 N}
\end{align*}
To prove Theorem \ref{thm:main} as well as Corollary \ref{cor:conj} we investigate and simplify each of the sums in the following.

\paragraph{The sum $S_3$. }
The first sum is straight forward to evaluate and we get:
\begin{align*}
S_3= - \sum_{i=2}^{N/2} \frac{8 i^3}{15 N} = \frac{-N^4-4 N^3-4 N^2+64}{120 N} = -\frac{N^3}{120} - \frac{N^2}{30}- \frac{N}{30} + \frac{8}{15N}
\end{align*}

\begin{remark}
This sum reveals the complexity of our task. Recall from the general formula for the discrepancy, that we need to divide $S_3+S_2+S_1+S_0$ by $N^2$ to obtain the final result. However, at the same time, recall that numerical results indicate that the expected discrepancy is $\frac{5}{72N}$. Hence, we see that the first two leading terms, i.e., $\mathrm{const}\cdot n^3$ and $\mathrm{const}\cdot n^2$ have to cancel in $S_3+S_2+S_1+S_0$. This is the reason why we need to develop all our approximations to the third leading term.
\end{remark}

\paragraph{The sum $S_2$.}
\begin{align*}
S_2&= \sum_{i=2}^{N/2} \frac{i^2 \left(-16 \sqrt{2} N
   \sqrt{\frac{i-1}{N}}+8 n \sqrt{\frac{i-1}{N}}
   \sqrt{\frac{i}{N}}+16 \sqrt{2} N \sqrt{\frac{i}{N}}+20\right)}{15
   N} \\
   &= \frac{20}{15N} \sum_{i=2}^{N/2} i^2 - \frac{16 \sqrt{2}}{15 \sqrt{N}} \sum_{i=2}^{N/2} i^2 \sqrt{i-1} +\frac{16 \sqrt{2}}{15 \sqrt{N}}  \sum_{i=2}^{N/2} i^{5/2} + \frac{8}{15N} \sum_{i=2}^{N/2} i^{5/2} \sqrt{i-1}
\end{align*}
Applying Lemma \ref{lem:sum} and \ref{lem:harmonic} we can simplify this expression to
$$S_2 =  \frac{N^3}{120} - \frac{N^{5/2}}{10\sqrt{N-2}} + \frac{52N^2}{225} + \frac{11N^{3/2}}{90\sqrt{N-2}}+\frac{113N}{360} + \mathcal{O}(N^{1/2}).$$

\paragraph{The sum $S_1$. }
\begin{align*}
S_1 &= \sum_{i=2}^{N/2} \frac{i \left(32 \sqrt{2} N \sqrt{\frac{i-1}{N}}-16 N
   \sqrt{\frac{i-1}{N}} \sqrt{\frac{i}{N}}-40 \sqrt{2} N
   \sqrt{\frac{i}{N}}\right)}{15 N}\\
&=\frac{1}{15 \sqrt{N}} \sum_{i=2}^{N/2} \left( i \left(32 \sqrt{2} \sqrt{i-1}-16
   \sqrt{i-1} \sqrt{\frac{i}{N}}-40 \sqrt{2}
   \sqrt{i}\right) \right) \\
   &= \frac{32 \sqrt{2}}{15 \sqrt{N}} \sum_{i=2}^{N/2} i \sqrt{i-1} - \frac{16}{15N}\sum_{i=2}^{N/2} i^{3/2} \sqrt{i-1} - \frac{8 \sqrt{2}}{3 \sqrt{N}} \left( -1 + \sum_{i=1}^{N/2} i^{3/2} \right )
\end{align*}
Applying Lemma \ref{lem:sum} and \ref{lem:harmonic} we can simplify this expression to
$$S_1= -\frac{22N^2}{225} +\frac{2N^{3/2}}{5 \sqrt{N-2}} - \frac{43N}{45} + \mathcal{O}(N^{1/2})$$

\paragraph{The sum $S_0$. }
\begin{align*}
S_0 &=\sum_{i=2}^{N/2} \frac{-16 \sqrt{2} N
   \sqrt{\frac{i-1}{N}}+8 N \sqrt{\frac{i-1}{N}}
   \sqrt{\frac{i}{N}}+10 \sqrt{2} N \sqrt{\frac{i}{N}}+15 N-5}{15 N} \\
   &= \sum_{i=2}^{N/2} \frac{15N-5}{N} - \frac{16\sqrt{2}}{15\sqrt{N}} \sum_{i=2}^{N/2} \sqrt{i-1} + \frac{10\sqrt{2}}{15 \sqrt{N}} \sum_{i=2}^{N/2} \sqrt{i} + \frac{8}{15N} \sum_{i=2}^{N/2} \sqrt{i}\sqrt{i-1}
\end{align*}
Applying Lemma \ref{lem:sum} and \ref{lem:harmonic} we can simplify this expression to
$$S_0=\frac{71 N}{90}-\frac{16}{45} \sqrt{2} \sqrt{\frac{N}{2}-1} \sqrt{N} + \mathcal{O}(N^{1/2})$$

\paragraph{Alltogether}

\begin{align*}
\sum_{i=2}^{N-1} Q_i &=S_3 + S_2 + S_1 + S_0 \\
&= -\frac{N^{5/2}}{10 \sqrt{N-2}}+\frac{N^2}{10}+\frac{47 N^{3/2}}{90 \sqrt{N-2}}+\frac{41
   N}{360}-\frac{16}{45} \sqrt{2} \sqrt{\frac{N}{2}-1} \sqrt{N} + \mathcal{O}(N^{1/2})
\end{align*}

Note that this can be further simplified:
\begin{align*}
-\frac{N^{5/2}} {10 \sqrt{N-2}}+\frac{N^2}{10} &= \frac{N^2}{10} \left(1 - \frac{\sqrt{N}}{\sqrt{N-2}} \right) \\
&= \frac{N^2}{10} \left(1 - \sqrt{1+\frac{2}{N-2}} \right) = \frac{N^2}{10} \left(1 - \left(1+\frac{1}{N-2} +\mathcal{O}(N^{-2}) \right)  \right) \\
&=-\frac{N}{10} + \mathcal{O}(1)
\end{align*}

Similarly,
$$\frac{47 N^{3/2}}{90 \sqrt{N-2}} = \frac{47N}{90} + \mathcal{O}(N^{-1}), $$
and 
$$-\frac{16}{45} \sqrt{2} \sqrt{\frac{N}{2}-1} \sqrt{N} = -\frac{16N}{45} + \mathcal{O}(1). $$

Therefore, we get
\begin{align} \label{sum:Q}
\notag \sum_{i=2}^{N-1} Q_i &=S_3 + S_2 + S_1 + S_0 \\
\notag &= -\frac{N}{10} +\frac{47N}{90} + \frac{41N}{360} - \frac{16N}{45} + \mathcal{O}(N^{1/2}) \\
&= \frac{13N}{72} + \mathcal{O}(N^{1/2})
\end{align}

\subsection{Proof of Theorem \ref{thm:main}}

From \eqref{form:main} and \eqref{sum:Q}
\begin{align*}
\mathbb{E}\mathcal{L}_2^2(\mathcal{P}_{\boldsymbol{\Omega}}) 
&= \frac{1}{4N} - \frac{1}{N^2} \left( 1-\frac{14 \sqrt{2}}{15 \sqrt{N}} + \frac{2}{5N} + \sum_{i=2}^{N-1} Q_i + \frac{1}{15N}\right)\\
&=\frac{1}{4N} - \frac{1}{N^2} \left(\frac{13N}{72} + \mathcal{O}(N^{1/2}) \right) \\
&= \frac{5}{72N} + \mathcal{O}(N^{-3/2})
\end{align*}

\begin{remark}
We can extend Theorem \ref{thm:main} to $N$ odd as follows. We can follow the lines of the proof for $N$ even and distinguish four cases for $Q_i$, i.e., $i=1$, $2\leq i < (n+1)/2$, $(n+1)/2<i<n-1$ and $i=n-1$. In addition, we need to add a separate analysis for the set $\Omega_{(n+1)/2}$. It will turn out that the contribution of this set is asymptotically of a lower order and, hence, the final result coincides with the case $N$ even.
\end{remark}

\section*{Acknowledgements}
I would like to thank Markus Kiderlen and Nathan Kirk for fruitful discussions as well as Stefan Steinerberger for help with Lemma \ref{lem:sum}. Furthermore, I wish to thank an anonymous reviewer for pointing out several inaccuracies in an earlier version of the manuscript.



\end{document}